\documentclass[12pt,reqno]{amsart}

\usepackage[a4paper,left=32mm,right=32mm,top=32mm,bottom=32mm]{geometry}
\usepackage{amsmath,amsthm,amssymb,mathtools}
\usepackage{microtype}
\usepackage[dvipsnames]{xcolor}
\usepackage[colorlinks=true,linkcolor=blue,citecolor=blue,urlcolor=blue]{hyperref}
\usepackage[capitalize,noabbrev]{cleveref}
\usepackage{enumitem}

\theoremstyle{plain}
\newtheorem{theorem}{Theorem}[section]
\newtheorem{lemma}[theorem]{Lemma}

\newtheorem{corollary}[theorem]{Corollary}
\theoremstyle{definition}

\newtheorem*{problem}{The lifting problem}
\theoremstyle{remark}
\newtheorem{remark}[theorem]{Remark}

\newcommand{\Pic}{\operatorname{Pic}}
\newcommand{\Aut}{\operatorname{Aut}}
\newcommand{\id}{\operatorname{id}}
\newcommand{\U}{\operatorname{U}}

\newcommand{\Ad}{\operatorname{Ad}}
\newcommand{\wG}{\smash{\widehat{G}}}
\newcommand{\wpic}{\smash{\widehat{p}}}
\newcommand{\wS}{\smash{\widehat{S}}}
\newcommand{\wsig}{\smash{\widehat{\sigma}}}
\newcommand{\wom}{\smash{\widehat{\omega}}}

\title[Lifting strongly graded rings]{Lifting strongly graded rings via overgroups}
\author{Emma Husen}
\address{Linnaeus university, 352 52 V\"axj\"o, Sweden}
\email{jh224jz@student.lnu.se}
\author{Stefan Wagner}
\address{Blekinge Institute of Technology, 371 79 Karlskrona, Sweden}
\email{stefan.wagner@bth.se}
\date{\today}

\begin{document}

\begin{abstract}
Motivated by lifting problems for principal bundles, we study the following algebraic problem. 
Let $\wG$ be a group with identity element $e$, let $G \leq \wG$, and let $S$ be a strongly $G$-graded unital ring with principal component $R:=S_e$. 
We ask whether the given grading extends to a strong $\wG$-grading without altering its $G$-homogeneous components. 
For every extension $\wpic$ of the Picard homomorphism of $S$, we construct a characteristic class in the relative third cohomology group
\[
    H^3_{\wpic}(\wG,G;\U(Z(R))),
\]
where $Z(R)$ denotes the center of $R$ and $\U(Z(R))$ its group of units.
The~vanishing of this class is equivalent to the existence of a lift. 
When it vanishes, the relative second cohomology group
\[
    H^2_{\wpic}(\wG,G;\U(Z(R)))
\]
acts simply transitively on the equivalence classes of lifts with Picard
homomorphism $\wpic$. 
We also give a proof of the criterion that an $H$-grading is strong if and only
if both its restriction to a normal subgroup $N \trianglelefteq H$ and its induced $H/N$-grading are strong. 
We illustrate the theory through a range of examples.
\end{abstract}

\maketitle

\enlargethispage{\baselineskip}
\section{Introduction}\label{sec:intro
}

Lifting problems for principal bundles arise naturally in geometry. 
A classical example is the construction of spin structures. 
For an oriented Riemannian manifold of dimension $n \geq 2$, one asks whether its oriented orthonormal frame bundle can be lifted from $\mathrm{SO}(n)$ to $\mathrm{Spin}(n)$ along the double covering $\mathrm{Spin}(n) \to \mathrm{SO}(n)$.
This is both an existence and a classification problem: a spin structure
exists precisely when the second Stiefel--Whitney class vanishes, and, when spin structures exist, their equivalence classes form a torsor under a first cohomology group; see, for example,~\cite[Chap.~2]{friedrich2000}.

A corresponding lifting problem for noncommutative principal bundles was studied in the $C^*$-algebraic setting by the second author in~\cite{Wa25}.

The present paper considers a purely algebraic lifting problem suggested by this picture. 
Strongly graded rings provide an algebraic model for free actions of discrete groups; see, for example,~\cite[Chap.~8]{Mont93}.
We ask whether a given strong~grading by a group $G$ can be extended to a strong grading by a larger group $\wG$ without altering the original homogeneous components.

More precisely, throughout the paper, let $\wG$ be a group with identity element $e$, let $G \leq \wG$, and let
\[
    S=\bigoplus_{g\in G}S_g
\]
be a strongly $G$-graded unital ring with principal component $R := S_e$. 

\begin{problem}
Determine whether there exists a strongly $\wG$-graded unital ring
\[
    \wS = \bigoplus_{x \in \wG} \wS_x
\]
containing $S$ as a unital subring and satisfying $\wS_g = S_g$ for all $g \in G$.
Such a ring is called a \emph{$\wG$-lift of $S$}. 
When such lifts exist, classify them up to $\wG$-graded ring isomorphisms whose restriction to $S$ is $\id_S$.
\end{problem}

For every $\wG$-lift of $S$,
\[
    \wS_e = R
    \qquad
    \text{and}
    \qquad
    \bigoplus_{g \in G} \wS_g = S.
\]
Thus, a lift adds homogeneous components indexed by $\wG \setminus G$ while
leaving the given $G$-graded ring unchanged.

The lifting problem has a natural formulation in terms of factor systems. 
The homogeneous components of a strongly graded ring are invertible bimodules over its principal component, while multiplication is encoded by compatible bimodule isomorphisms. 
Consequently, a lift of $S$ corresponds to an extension of its factor system from $G$ to $\wG$.

This reformulation separates the lifting problem into two successive questions:
\begin{enumerate}
    \item Does the Picard homomorphism of $S$, $p: G \to \Pic(R)$, $g \mapsto [S_g]$, extend to a group homomorphism $\wpic: \wG \to \Pic(R)$?
    \item For such an extension $\wpic$, can the representatives $M_x$ of the classes $\wpic(x)$, $x \in \wG$, be chosen with $M_g = S_g$ for all $g \in G$ and be equipped with associative multiplication maps extending those of $S$?
\end{enumerate}

The first question gives an immediate obstruction at the level of the Picard~group. To address the second, let $Z(R)$ denote the center of $R$ and let $\U(Z(R))$ denote its group of units. 
For each extension $\wpic$, we construct a characteristic class
\[
    \kappa_S(\wpic)
    \in
    H^3_{\wpic}(\wG,G;\U(Z(R))).
\]
This class measures the failure of chosen multiplication maps to be associative while remaining fixed on $G$. 
We prove that a $\wG$-lift of $S$ with Picard homomorphism $\wpic$ exists if and only if $\kappa_S(\wpic)=0$. 
When this class vanishes, the relative cohomology group
\[
    H^2_{\wpic}(\wG,G;\U(Z(R)))
\]
acts simply transitively on the equivalence classes of such lifts.

We also use a complementary recognition criterion for strong gradings over normal subgroups. 
More precisely, if an $H$-graded ring is restricted to a normal subgroup $N \trianglelefteq H$ and simultaneously coarsened to an $H/N$-grading, then the original grading is strong if and only if both associated gradings are strong. 
This result was obtained independently by the first author in her master's thesis. 
After the first version of the present paper appeared on arXiv, Roozbeh Hazrat kindly pointed out to us that the same result had previously been proved, in the more general setting of graded rings with graded local units, in \cite[Lem.~2.4]{ClHaRi19}. 
The criterion is particularly useful for candidate graded extensions arising independently of the factor-system construction.


The paper is organized as follows. 
In Section~\ref{sec:prelim}, we recall strongly graded rings, invertible bimodules, factor systems, and the associated group-cohomological obstruction theory. 
In Section~\ref{sec:lifting}, we formulate the lifting problem in terms of factor systems and prove the obstruction and classification results; see~\cref{thm:lift-factor-correspondence,thm:obstruction,thm:classification}.
In Section~\ref{sec:normal-subgroups}, we prove the recognition criterion for strong~gradings over normal subgroups. 
Finally,~\cref{sec:examples} illustrates the theory through cyclic overgroups, the double covering $\mathrm{Spin}(2) \to \mathrm{SO}(2)$, crossed products, Picard obstructions, twisted group
rings, and semidirect-product lifts.

A natural direction for future work is to formulate and study an analogous
lifting problem for Hopf--Galois extensions, with the present theory of
strongly graded rings serving as a first model case.

\section{Preliminaries}\label{sec:prelim}

All rings are associative and unital, all ring homomorphisms are unital, and
all module actions are unital. 
No commutativity assumptions are imposed. 

Let $H$ be a group with identity element $e$, and let $A$ be a ring. 
An $H$-grading on $A$ is a family $(A_h)_{h\in H}$ of additive subgroups of $A$
such that
\[
    A = \bigoplus_{h \in H} A_h
\]
and $A_h A_k \subseteq A_{hk}$ for all $h,k \in H$.
The subgroups $A_h$, $h \in H$, are called the \emph{homogeneous components} of the grading. 
The component $A_e$ is called the \emph{principal~component}. 
The grading is called \emph{strong} if $A_h A_k = A_{hk}$ for all $h,k \in H$.

The identity of an $H$-graded unital ring belongs to its principal component.
If the grading is strong, then $A_h A_{h^{-1}} = A_e = A_{h^{-1}} A_h$ for all $h \in H$.
In particular, there exist finite families $u_i \in A_h$ and $v_i \in A_{h^{-1}}$ such that
\[
    \sum_i u_iv_i=1.
\]

For the discussion of invertible bimodules and factor systems below, the principal component $A_e$ is assumed to be the fixed ring $R$. 
This includes the given grading when $(H,A) = (G,S)$ and any possible lift when $(H,A) = (\wG,\wS)$.

An $R$-bimodule $M$ is called \emph{invertible} if there exists an $R$-bimodule $N$ such that
\[
    M \otimes_R N \cong R
    \qquad
    \text{and}
    \qquad
    N \otimes_R M \cong R
\]
as $R$-bimodules. 
The isomorphism class of an invertible $R$-bimodule $M$ is denoted by $[M]$.
These classes form the \emph{Picard group} $\Pic(R)$, with multiplication given by
\[
    [M][N] : =[M \otimes_R N].
\]
The identity element is $[R]$, and the inverse of $[M]$ is represented by any
inverse bimodule of $M$.

The following standard result connects strong gradings with invertible bimodules:

\begin{lemma}[{see, e.g.,~\cite[Prop.~I.3.6]{NuOy04}}]\label{lem:mult-iso}
Let $A = \bigoplus_{h \in H} A_h$ be a strongly $H$-graded ring with principal component $R$. 
For all $h,k\in H$, multiplication induces an $R$-bimodule~isomorphism
\[
    m_{h,k}: A_h \otimes_R A_k \to A_{hk},
    \qquad
    a \otimes b \mapsto ab.
\]
Consequently, $A_h$ is an invertible $R$-bimodule with inverse $A_{h^{-1}}$.
\end{lemma}

It follows from Lemma~\ref{lem:mult-iso} that every strongly $H$-graded ring $A$ with principal component $R$ determines a group homomorphism
\begin{equation}\label{eq:picard-map}
    p: H \to \Pic(R),
    \qquad
    h \mapsto [A_h],
\end{equation}
called the \emph{Picard homomorphism} of $A$.

The homogeneous components of a strongly graded ring do not determine the
ring by themselves; their multiplication maps must also be recorded. 

A \emph{normalized factor system} for $(H,R)$ consists of a family $(M_h)_{h \in H}$ of invertible $R$-bimodules, with $M_e = R$, together with $R$-bimodule isomorphisms
\[
    \mu_{h,k}: M_h \otimes_R M_k \to M_{hk},
    \qquad
    h,k\in H.
\]
The maps $\mu_{e,h}$ and $\mu_{h,e}$ are required to be the canonical module actions, and
\begin{equation}\label{eq:associativity}
    \mu_{hk,\ell} \circ (\mu_{h,k} \otimes \id_{M_\ell})
    =
    \mu_{h,k\ell} \circ (\id_{M_h} \otimes \mu_{k,\ell})
\end{equation}
is required to hold for all $h,k,\ell\in H$. 
We abbreviate these families by $M=(M_h)_{h \in H}$ and
$\mu=(\mu_{h,k})_{h,k \in H}$.

Given a normalized factor system $(M,\mu)$ for $(H,R)$, the Abelian group
\[
    A(M,\mu) := \bigoplus_{h \in H}M_h
\]
becomes a unital ring when the multiplication is defined on homogeneous elements by
\[
    m_hm_k := \mu_{h,k}(m_h \otimes m_k)
\]
for all $h,k \in H$, $m_h \in M_h$, and $m_k \in M_k$, and then extended bilinearly.
Condition~\eqref{eq:associativity} is precisely the associativity of this multiplication, while the normalization conditions ensure that $1_R \in M_e$ is its identity. 
Since every $\mu_{h,k}$ is surjective, the resulting $H$-grading is strong.

Conversely, let $A = \bigoplus_{h \in H} A_h$ be a strongly $H$-graded ring with principal~component $R$. 
By Lemma~\ref{lem:mult-iso}, the homogeneous components $(A_h)_{h \in H}$, together with the multiplication maps
\[
    m_{h,k}: A_h \otimes_ R A_k \to A_{hk}
\]
for all $h,k \in H$, form a normalized factor system.

Two normalized factor systems $(M,\mu)$ and $(N,\nu)$ for $(H,R)$ are called~\emph{isomorphic} if there exists a family of $R$-bimodule isomorphisms $\varphi_h: M_h \to N_h$, $h \in H$, such that $\varphi_e = \id_R$ and
\[
    \varphi_{hk} \circ \mu_{h,k}
    =
    \nu_{h,k} \circ (\varphi_h \otimes \varphi_k)
\]
for all $h,k \in H$.

The preceding constructions yield the following correspondence:

\begin{theorem}\label{thm:factor-correspondence}
There is a bijection between:
\begin{enumerate}[label=(\alph*)]
    \item graded-isomorphism classes of strongly $H$-graded rings with principal component $R$, where the graded isomorphisms restrict to the identity on $R$;
    \item isomorphism classes of normalized factor systems for $(H,R)$.
\end{enumerate}
\end{theorem}

The Picard group acts naturally on the center of $R$. 
For every invertible $R$-bimodule $M$, there exists a unique automorphism
\[
    \Delta_M \in \Aut(Z(R))
\]
characterized by $m z = \Delta_M(z) m$ for all $m \in M$ and $z \in Z(R)$; see, for example, \cite[Lem.~I.3.11.2]{NuOy04}.
The automorphism $\Delta_M$ depends only on the isomorphism class of~$M$, and the assignment
\[
    \Delta: \Pic(R) \to \Aut(Z(R)),
    \qquad
    [M] \mapsto \Delta_M,
\]
is a group homomorphism, called the \emph{Fröhlich map}.

Consequently, every group homomorphism $p:H \to \Pic(R)$ induces an action of $H$ on $\U(Z(R))$ given by
\begin{equation}\label{eq:froehlich}
    {}^hu = \Delta_{p(h)}(u),
    \qquad
    h \in H, u \in \U(Z(R)).
\end{equation}
The corresponding group cohomology groups are denoted by
\[
    H^n_p(H,\U(Z(R))),
    \qquad
    n \geq 0,
\]
where the subscript $p$ records the induced action; see, for example,~\cite[Chap.~IV]{MacLane95} for background on group cohomology.

Given a group homomorphism $p:H \to \Pic(R)$, a normalized factor system $(M,\mu)$ for $(H,R)$ is said to \emph{realize} $p$ if $[M_h]=p(h)$ for all $h \in H$.
The obstruction to the existence of such a factor system is measured by a characteristic class
\[
    \kappa(p) \in H^3_p (H,\U(Z(R))).
\]
More precisely, the following standard result holds:

\begin{theorem}[{see, e.g.,~\cite[Prop.~I.3.18]{NuOy04}}]\label{thm:factor-system-cohomology}
Let $p:H \to \Pic(R)$ be a group~homomorphism.
\begin{enumerate}[label=(\roman*)]
    \item There exists a normalized factor system realizing $p$ if and only
    if $\kappa(p) = 0$.
    \item If $\kappa(p)=0$, then $H^2_p(H,\U(Z(R)))$ acts simply transitively on the isomorphism classes of normalized factor systems realizing $p$.
\end{enumerate}
\end{theorem}

\section{The lifting problem}\label{sec:lifting}

The canonical normalized factor system of $S$ consists of the family $(S_g)_{g\in G}$ together with the multiplication maps
\[
    m_{g,h}: S_g \otimes_R S_h \to S_{gh},
    \qquad
    g,h \in G.
\]
By a slight abuse of notation, we denote this factor system by $(S,m)$.

A normalized factor system $(M,\mu)$ for $(\wG,R)$ is said to \emph{extend $(S,m)$} if $M_g = S_g$ and $\mu_{g,h} = m_{g,h}$ for all $g,h\in G$.

Two extensions $(M,\mu)$ and $(N,\nu)$ of $(S,m)$ are called \emph{equivalent} if there exists an isomorphism of factor systems $\varphi: (M,\mu) \to (N,\nu)$ such that $\varphi_g = \id_{S_g}$ for all $g \in G$.

The factor-system correspondence in Theorem~\ref{thm:factor-correspondence} admits the following relative version:

\begin{theorem}\label{thm:lift-factor-correspondence}
There is a bijection between:
\begin{enumerate}[label=(\alph*)]
    \item equivalence classes of $\wG$-lifts of $S$;
    \item equivalence classes of normalized factor systems for $(\wG,R)$
    extending $(S,m)$.
\end{enumerate}
\end{theorem}
\begin{proof}
Every $\wG$-lift of $S$ determines a normalized factor system extending $(S,m)$, and every such extension determines a $\wG$-lift of $S$. 
Under this correspondence, graded isomorphisms restricting to the identity on $S$ correspond precisely to~isomorphisms of factor systems restricting to $\id_{S_g}$ for all $g\in G$. 
Hence the correspondence induces the asserted bijection on equivalence classes.
\end{proof}

The correspondence yields the following necessary condition for the existence of a lift.

\begin{corollary}\label{cor:picard-necessary}
Suppose that $S$ admits a $\wG$-lift.
Then its Picard homomorphism $p: G \to \Pic(R)$ extends to a group homomorphism $\wpic: \wG \to \Pic(R)$.
\end{corollary}

The converse need not hold, since an extension of $p$ determines only the isomorphism classes of the homogeneous components, not compatible multiplication maps extending those of $S$.

\begin{remark}\label{rem:lift-by-zero}
A lift cannot be obtained by setting $\wS_x=0$ for $x\notin G$. 
Indeed, strongness would require $\wS_x \wS_{x^{-1}} = \wS_e = R$ for all $x \in \wG$. 
Thus, when $R \neq 0$, every homogeneous component must be nonzero. 
In fact, by Lemma~\ref{lem:mult-iso}, the additional components must be invertible $R$-bimodules.
\end{remark}

\subsection{The associativity obstruction}

For the remainder of this section, fix a group homomorphism $\wpic: \wG \to \Pic(R)$ extending the Picard homomorphism $p: G \to \Pic(R)$ of $S$.
For each $x \in \wG$, choose an invertible $R$-bimodule $M_x$ representing $\wpic(x)$, with $M_g = S_g$ for all $g\in G$. 
In particular, $M_e = R$.

Since $\wpic$ is a homomorphism,
\[
    [M_x \otimes_R M_y]
    =
    \wpic(x) \wpic(y)
    =
    \wpic(xy)
    =
    [M_{xy}]
\]
for all $x,y \in \wG$. 
Hence there exist $R$-bimodule isomorphisms
\[
    \mu_{x,y}: M_x \otimes_R M_y \to M_{xy},
    \qquad
    x,y \in \wG.
\]
Choose these isomorphisms subject to $\mu_{g,h} = m_{g,h}$ for all $g,h \in G$, and require $\mu_{e,x}$ and $\mu_{x,e}$ to be the canonical module actions for each $x \in \wG$. 
Such a choice is possible: the two requirements agree on their overlap because $(S,m)$ is normalized, and for all remaining pairs the required isomorphisms exist by the preceding argument.
The resulting maps extend the multiplication maps of $S$, but need not satisfy the associativity condition~\eqref{eq:associativity}.

For $x,y,z \in \wG$, the two composites from $M_x\otimes_R M_y\otimes_R M_z$ to $M_{xyz}$ differ by an automorphism of $M_{xyz}$. 
Since every $R$-bimodule automorphism of an invertible $R$-bimodule is given by left multiplication by a unique element of $\U(Z(R))$ (see, e.g.,~\cite[Lem.~I.3.11.2]{NuOy04}), there exists a unique element
\[
    \alpha(x,y,z) \in \U(Z(R))
\]
such that
\begin{equation}\label{eq:associator}
    \mu_{xy,z} \circ (\mu_{x,y} \otimes \id_{M_z})
    =
    \alpha(x,y,z) \mu_{x,yz} \circ (\id_{M_x} \otimes \mu_{y,z}).
\end{equation}

The normalization of the maps $\mu_{x,y}$ implies that $\alpha$ is normalized. 
Moreover, since $\mu_{g,h} = m_{g,h}$ for all $g,h \in G$ and multiplication in $S$ is associative, $\alpha\vert_{G^3} = 1$. 
The natural receptacle for this obstruction is therefore relative group cohomology; see \cite[Chap.~XI, Sec.~9]{MacLane95} or~\cite{Tak59}.

As in Section~\ref{sec:prelim}, the homomorphism $\wpic: \wG \to \Pic(R)$ makes $\U(Z(R))$ a $\wG$-module through the Fröhlich map. 
Let
\[
    C^n_{\wpic}(\wG,\U(Z(R)))
\]
denote the group of normalized $n$-cochains for this action. 
Restriction to $G$ defines a morphism of cochain complexes
\[
    \operatorname{res}_G^{\wG}:
    C^\bullet_{\wpic}(\wG,\U(Z(R)))
    \to
    C^\bullet_p(G,\U(Z(R))).
\]
Since restriction commutes with the group-cohomology differential, its kernel is a subcomplex. 
The relative cochain complex is defined by
\[
    C^\bullet_{\wpic}(\wG,G;\U(Z(R)))
    :=
    \ker(\operatorname{res}_G^{\wG}).
\]
Thus, for $n\geq 1$,
\[
    C^n_{\wpic}(\wG,G;\U(Z(R)))
    =
    \{c\in C^n_{\wpic}(\wG,\U(Z(R))): c\vert_{G^n} = 1\}.
\]
Its cohomology is denoted by
\[
    H^n_{\wpic}(\wG,G;\U(Z(R))),
    \qquad
    n \geq 1.
\]

\begin{lemma}\label{lem:associator-cocycle}
The map $\alpha$ defined by \eqref{eq:associator} is a normalized relative $3$-cocycle. 
Its cohomology class
\[
    [\alpha]\in H^3_{\wpic}(\wG,G;\U(Z(R)))
\]
is independent of the choices of $M_x$, $x \in \wG$, and $\mu_{x,y}$, $x,y \in \wG$.
\end{lemma}
\begin{proof}
For $x,y,z,t \in \wG$, let
\[
\begin{aligned}
    L_{x,y,z,t}
    &:=
    \mu_{xyz,t}
    \circ
    (\mu_{xy,z} \otimes \id_{M_t})
    \circ
    (\mu_{x,y} \otimes \id_{M_z} \otimes \id_{M_t}),
    \\
    R_{x,y,z,t}
    &:=
    \mu_{x,yzt}
    \circ
    (\id_{M_x} \otimes \mu_{y,zt})
    \circ
    (\id_{M_x} \otimes \id_{M_y} \otimes \mu_{z,t}).
\end{aligned}
\]
Repeated application of \eqref{eq:associator} along one side of the associativity pentagon gives
\[
    L_{x,y,z,t}
    =
    {}^x\alpha(y,z,t)
    \alpha(x,yz,t)
    \alpha(x,y,z)
    R_{x,y,z,t}.
\]
Here the term ${}^x\alpha(y,z,t)$ arises from moving $\alpha(y,z,t)$ past the factor $M_x$ using \eqref{eq:froehlich}.
Applying \eqref{eq:associator} along the other side gives
\[
    L_{x,y,z,t}
    =
    \alpha(xy,z,t)
    \alpha(x,y,zt)
    R_{x,y,z,t}.
\]
Since $R_{x,y,z,t}$ is an isomorphism, it follows that
\[
    {}^x\alpha(y,z,t)
    \alpha(x,yz,t)
    \alpha(x,y,z)
    =
    \alpha(xy,z,t)
    \alpha(x,y,zt)
\]
for all $x,y,z,t \in \wG$. 
This is precisely the cocycle identity $d\alpha=1$.
The normalization of the maps $\mu_{x,y}$ implies that $\alpha$ is normalized. 
Moreover, $\alpha\vert_{G^3}=1$, so $\alpha$ is a normalized relative $3$-cocycle.

It remains to prove independence of the choices. 
First suppose that the bimodules $M_x$, $x \in \wG$, are fixed, and let
$(\mu'_{x,y})_{x,y \in \wG}$ be another admissible family of isomorphisms. 
For all $x,y \in \wG$, there is a unique $c(x,y) \in \U(Z(R))$ such that
\[
    \mu'_{x,y} = c(x,y) \mu_{x,y}.
\]
Since both families are normalized and agree with $m_{g,h}$ on $G \times G$, the map
\[
    c: \wG \times \wG \to \U(Z(R))
\]
is a normalized relative $2$-cochain.

Let $\alpha'$ be the associator determined by $\mu'$. 
Substitution into~\eqref{eq:associator} gives
\[
    \alpha'(x,y,z)
    =
    ({}^xc(y,z))^{-1}
    c(x,yz)^{-1}
    c(xy,z)
    c(x,y)
    \alpha(x,y,z)
\]
for all $x,y,z \in \wG$. 
Equivalently, $\alpha' = (dc)^{-1} \alpha$.
Thus $\alpha$ and $\alpha'$ determine the same class in $H^3_{\wpic}(\wG,G;\U(Z(R)))$.

Finally, suppose that another family of representatives $(M'_x)_{x\in\wG}$ is chosen, together with multiplication maps $\nu_{x,y}$, $x,y \in \wG$. 
Choose $R$-bimodule isomorphisms
\[
    \varphi_x: M_x \to M'_x,
    \qquad
    x \in \wG,
\]
with $\varphi_g = \id_{S_g}$ for all $g\in G$. 
Transport the maps $\nu_{x,y}$ to the bimodules $M_x$ by setting
\[
    \widetilde\mu_{x,y}
    :=
    \varphi_{xy}^{-1} \circ \nu_{x,y} \circ (\varphi_x \otimes \varphi_y).
\]
The associator is unchanged under this transport, and the comparison between $\widetilde\mu$ and $\mu$ reduces to the preceding case. 
Hence the class $[\alpha]$ is independent of all choices.
\end{proof}

The class
\[
    \kappa_S(\wpic)
    :=
    [\alpha]
    \in
    H^3_{\wpic}(\wG,G;\U(Z(R)))
\]
is called the \emph{characteristic class} of $\wpic$ with respect to $(S,m)$.

\begin{theorem}\label{thm:obstruction}
There exists a $\wG$-lift of $S$ with Picard homomorphism $\wpic$ if and only if $\kappa_S(\wpic) = 0$.
\end{theorem}
\begin{proof}
If such a lift exists, its associative multiplication maps give $\alpha = 1$, and hence $\kappa_S(\wpic)=0$.

Conversely, suppose that $\kappa_S(\wpic)=0$. 
Then $\alpha = dc$ for some normalized relative $2$-cochain $c$. 
Define
\[
    \mu'_{x,y} := c(x,y)\mu_{x,y},
    \qquad
    x,y \in \wG.
\]
By the transformation formula in the proof of Lemma~\ref{lem:associator-cocycle}, the corresponding associator is $\alpha' = (dc)^{-1}\alpha = 1$.
Moreover, since $c\vert_{G^2} = 1$, the maps $\mu'_{g,h}$ agree with $m_{g,h}$ for all $g,h\in G$. 
Thus $(M,\mu')$ is a normalized factor system extending $(S,m)$, and Theorem~\ref{thm:lift-factor-correspondence} yields the required lift.
\end{proof}

\subsection{Classification for fixed Picard data}

Assume that $\kappa_S(\wpic) = 0$ and fix a normalized factor system $(M,\mu)$ extending $(S,m)$ and realizing $\wpic$. 
For
\[
    c \in Z^2_{\wpic}(\wG,G;\U(Z(R))),
\]
define
\begin{equation}\label{eq:twist}
    \mu^c_{x,y} := c(x,y) \mu_{x,y},
    \qquad
    x,y \in \wG.
\end{equation}
Since $dc = 1$, the maps $\mu^c_{x,y}$ are associative, and since $c\vert_{G^2} = 1$, they still extend the multiplication maps of $S$.

\begin{theorem}\label{thm:classification}
The group
\[
    H^2_{\wpic}(\wG,G;\U(Z(R)))
\]
acts simply transitively on the equivalence classes of $\wG$-lifts of $S$ with Picard homomorphism $\wpic$.
\end{theorem}
\begin{proof}
By Theorem~\ref{thm:lift-factor-correspondence}, it suffices to consider equivalence classes of normalized factor systems extending $(S,m)$.
Twisting as in~\eqref{eq:twist} defines an action of relative $2$-cocycles on these factor systems. 
If two cocycles differ by a relative $2$-coboundary, multiplication by the corresponding relative $1$-cochain defines an equivalence between the resulting factor systems.
Hence the action descends to
\[
    H^2_{\wpic}(\wG,G;\U(Z(R))).
\]

To prove transitivity, let $(M,\mu)$ and $(N,\nu)$ be normalized factor systems extending $(S,m)$ and realizing $\wpic$. 
Choose $R$-bimodule isomorphisms
\[
    \varphi_x: M_x \to N_x,
    \qquad
    x \in \wG,
\]
with $\varphi_g = \id_{S_g}$ for all $g \in G$. 
Transport the maps $\nu_{x,y}$, $x,y \in \wG$ to the bimodules $M_x$ by setting
\[
    \widetilde\nu_{x,y}
    :=
    \varphi_{xy}^{-1} \circ \nu_{x,y} \circ(\varphi_x \otimes \varphi_y)
    :
    M_x \otimes_R M_y \to M_{xy}.
\]
There is a unique normalized relative $2$-cochain $c$ such that $\widetilde\nu_{x,y} = c(x,y)\mu_{x,y}$ for all $x,y \in \wG$.
Since both $\mu$ and $\widetilde\nu$ are associative, $dc = 1$. 
Hence $(M,\widetilde\nu) = (M,\mu^c)$, while $\varphi$ defines an isomorphism from $(M,\widetilde\nu)$ to $(N,\nu)$. 
Thus the action is transitive.

To prove freeness, suppose that $(M,\mu^c)$ is equivalent to $(M,\mu)$. 
The components of such an equivalence have the form
\[
    \varphi_x = b(x)\id_{M_x},
    \qquad
    x \in \wG,
\]
where $b$ is a normalized relative $1$-cochain. 
Compatibility with the multiplication maps gives
\[
    c(x,y)
    =
    b(x) {}^xb(y) b(xy)^{-1}
    =
    (db)(x,y)
\]
for all $x,y \in \wG$. 
Thus $c$ is a relative coboundary, and the action is free.
\end{proof}

Combining \cref{thm:obstruction,thm:classification} gives a complete cohomological solution to the lifting problem posed in the introduction:
first extend the Picard homomorphism $p$ of $S$ to a homomorphism $\wpic: \wG \to \Pic(R)$, then test the characteristic class
\[
    \kappa_S(\wpic)
    \in
    H^3_{\wpic}(\wG,G;\U(Z(R))),
\]
and, when it vanishes, classify the lifts using
\[
    H^2_{\wpic}(\wG,G;\U(Z(R))).
\]

\section{Strong gradings and normal subgroups: a recognition criterion}
\label{sec:normal-subgroups}

The preceding section addresses the existence and classification of lifts.
We now turn to the complementary problem of determining when a given graded extension is strong. 
For normal subgroups, this can be checked in two stages: first on the subgroup and then on the quotient.

Let $H$ be a group with identity element $e$, let $A = \bigoplus_{h\in H} A_h$ be an $H$-graded ring, and let $N \trianglelefteq H$. 
There are two naturally associated gradings. 
The \emph{subgroup grading} is the $N$-graded subring
\[
    A_N := \bigoplus_{n\in N} A_n.
\]
The \emph{quotient grading} is the $H/N$-grading of $A$ with homogeneous components
\[
    A_{[x]} := \bigoplus_{n \in N} A_{xn},
    \qquad
    [x] \in H/N.
\]
The normality of $N$ ensures that $A_{[x]} A_{[y]} \subseteq A_{[xy]}$ for all $x,y \in H$.

If the original $H$-grading is strong, then both associated gradings are strong; see, for example,~\cite[Prop.~2.1]{Lann21}. 
The converse also holds.
After the first version of this paper appeared on arXiv, Roozbeh Hazrat kindly pointed out to us that this result had already been proved in the more general setting of graded rings with graded local units in \cite[Lem.~2.4]{ClHaRi19}. 
We were not aware of this reference when the first version was posted. 
The same result was obtained independently by the first author in her master's thesis~\cite[Thm.~6]{Hus26} for the unital setting considered here. 
For completeness, we include a proof adapted to our notation.

\pagebreak[3]
\begin{theorem}\label{thm:normal-permanence}
The following are equivalent:
\begin{enumerate}[label=(\alph*)]
    \item the $H$-grading of $A$ is strong;
    \item the subgroup grading of $A_N$ is strongly $N$-graded and the quotient grading is strongly $H/N$-graded.
\end{enumerate}
\end{theorem}
\begin{proof}
We only prove \textup{(b)}$\Rightarrow$\textup{(a)}, since \textup{(a)}$\Rightarrow$\textup{(b)} was recalled above. 
Assume that the subgroup and quotient gradings are strong.
Fix $x,y \in H$. 
It suffices to establish that $A_{xy} \subseteq A_xA_y$, because the reverse inclusion follows from the grading. 
By the strongness of the quotient grading,
\[
    A_{xy} \subseteq A_{[xy]}=A_{[x]}A_{[y]}.
\]
Expanding the quotient components and projecting onto $A_{xy}$ gives
\[
    A_{xy}
    \subseteq
    \sum_{\substack{n,m \in N\\xnym=xy}}
    A_{xn}A_{ym}.
\]
For each $n \in N$, the condition $xnym = xy$ determines $m = y^{-1}n^{-1}y \in N$, where normality of $N$ is used. 
Since $ym = n^{-1}y$, it follows that
\[
    A_{xy}
    \subseteq
    \sum_{n \in N} A_{xn} A_{n^{-1}y}.
\]

Since the subgroup grading is strong, $A_n A_{n^{-1}} = A_e = A_{n^{-1}} A_n$.
It follows that $A_{xn} = A_xA_n$ and $A_{n^{-1}y} = A_{n^{-1}}A_y$.
Indeed,
\[
    A_{xn}
    =
    A_{xn}A_e
    =
    A_{xn}A_{n^{-1}}A_n
    \subseteq
    A_xA_n,
\]
and the reverse inclusion follows from the grading; the second equality is proved similarly. 
Consequently,
\[
    A_{xn}A_{n^{-1}y}
    =
    A_xA_nA_{n^{-1}}A_y
    =
    A_xA_y.
\]
Thus every contribution of degree $xy$ to $A_{[x]}A_{[y]}$ belongs to $A_xA_y$, proving that $A_{xy}\subseteq A_xA_y$. 
Hence the original $H$-grading is strong.
\end{proof}

For our lifting problem, the theorem gives the following recognition criterion.

\begin{corollary}\label{cor:normal-lift}
Assume that $G\trianglelefteq\wG$, and let $\wS = \bigoplus_{x \in \wG} \wS_x$ be a $\wG$-graded ring whose $G$-graded subring is the given strongly $G$-graded ring $S$. 
Then $\wS$ is a $\wG$-lift of $S$ if and only if its quotient grading over $\wG/G$ is strong.
\end{corollary}

\begin{proof}
The subgroup grading is the given strongly $G$-graded ring $S$. The
claim therefore follows from \cref{thm:normal-permanence} with
$H=\wG$ and $N=G$.
\end{proof}

The corollary is useful when a $\wG$-graded extension of $S$ arises from an independent construction, for example from generators and relations, a crossed product, or an ambient graded ring. 
It does not construct such an extension; rather, it reduces the verification of strongness to the quotient $\wG/G$-grading. 
Normality also allows the relative cohomology groups from the preceding section to be studied through the Lyndon--Hochschild--Serre spectral sequence, although they do not in general reduce to the cohomology of $\wG/G$.

\section{Examples}\label{sec:examples}

We conclude with examples illustrating the different layers of the lifting problem. 
We begin with cyclic overgroups and then interpret the classical double covering $\mathrm{Spin}(2) \to \mathrm{SO}(2)$ as a lift from a strong $2\mathbb Z$-grading to a strong $\mathbb Z$-grading. 
We next specialize the obstruction theory to crossed products, exhibit a Picard obstruction, and identify the characteristic class of a twisted group ring with the familiar obstruction to extending a $2$-cocycle. 
Finally, we use the normal-subgroup recognition theorem to construct semidirect-product lifts.

\subsection{Cyclic overgroups}

Let $m\geq 2$, and let $S = \bigoplus_{n \in m\mathbb Z} S_n$ be a strongly $m\mathbb Z$-graded ring with $S_0=R$ and Picard homomorphism $p:m\mathbb Z \to \Pic(R)$.
The homomorphism $p$ extends to $\mathbb Z$ if and only if $p(m)$ admits an
$m$th root in $\Pic(R)$. 
Indeed, an extension $\wpic$ is determined by $\wpic(1)$, which must satisfy $\wpic(1)^m = p(m)$.

\begin{lemma}\label{lem:cyclic-lifts}
Suppose that $\wpic: \mathbb Z \to \Pic(R)$ extends $p$. 
Then there exists a $\mathbb Z$-lift of $S$ with Picard homomorphism $\wpic$. 
Its equivalence classes form a torsor under
\[
    \operatorname{coker}
    \!
    \left(
        H^1_{\wpic}(\mathbb Z,\U(Z(R))) \to H^1_p(m\mathbb Z,\U(Z(R)))
    \right).
\]
\end{lemma}
\begin{proof}
Since both groups have cohomological dimension one, the long exact sequence of the pair shows that the relative third cohomology vanishes and identifies $H^2_{\wpic}(\mathbb Z,m\mathbb Z;\U(Z(R)))$ with the indicated cokernel.
The claim now follows from~\cref{thm:obstruction,thm:classification}.
\end{proof}

If the action on $\U(Z(R))$ induced by $\wpic$ is trivial, then
\[
    H^2\bigl(\mathbb Z,m\mathbb Z;\U(Z(R))\bigr)
    \cong
    \U(Z(R))/\U(Z(R))^m.
\]
Consequently, for the fixed Picard homomorphism $\wpic$, the lift is unique up
to equivalence if and only if every central unit of $R$ admits an $m$th root.

\subsection{The double covering \texorpdfstring{$\mathrm{Spin}(2)\to\mathrm{SO}(2)$}{Spin(2)→SO(2)}}

The geometric motivation for the lifting problem is already visible in dimension two. 
For an oriented Riemannian surface, the oriented orthonormal frame bundle is a principal $\mathrm{SO}(2)$-bundle, and a spin structure is a lift of this bundle along the double covering $q: \mathrm{Spin}(2) \to \mathrm{SO}(2)$.
The simplest instance is obtained by taking the base space to be a point: $\mathrm{SO}(2)$ is then a principal $\mathrm{SO}(2)$-bundle, and $\mathrm{Spin}(2)$ provides its lift. 
The following example is the algebraic counterpart of this basic fibrewise model.

Under the standard identifications $\mathrm{Spin}(2)\cong\mathbb T\cong\mathrm{SO}(2)$,
the double covering is given by $q(z)=z^2$. 
The induced map on character groups is $q^*: \mathbb Z \to \mathbb Z$, $n \mapsto 2n$.
Thus, on the level of character groups, the covering corresponds to the
inclusion $2\mathbb Z\leq\mathbb Z$.

Passing to the algebras of representative functions, let
\[
    \wS : =\mathbb C[u,u^{-1}]
\]
with its canonical strong $\mathbb Z$-grading $\wS_n := \mathbb C u^n$, $n \in \mathbb Z$.
The pullback along $q$ identifies the algebra corresponding to $\mathrm{SO}(2)$ with the subring
\[
    S := \mathbb C[u^2,u^{-2}] = \bigoplus_{n \in 2\mathbb Z} \mathbb C u^n,
\]
This is a strongly $2\mathbb Z$-graded ring with $S_n = \mathbb C u^n$, $n \in 2 \mathbb Z$, and $\wS$ is a $\mathbb Z$-lift of $S$. 
Thus a covering of compact groups becomes an extension of strong gradings after passing to character groups.

The Picard homomorphism of $S$ is trivial, and we consider its trivial
extension to $\mathbb Z$. Applying \cref{lem:cyclic-lifts} with $m=2$ gives
\[
    H^2(\mathbb Z,2\mathbb Z;\mathbb C^\times)
    \cong
    \mathbb C^\times/(\mathbb C^\times)^2
    =
    0.
\]
Consequently, $\mathbb C[u,u^{-1}]$ is, up to equivalence, the unique $\mathbb Z$-lift of
$S$ with trivial Picard homomorphism.

\subsection{Crossed products}

Suppose that every homogeneous component $S_g$ contains a unit of $S$.
Choose a unit $u_g\in S_g$ for each $g\in G$, with $u_e=1$. 
Then $S_g = Ru_g = u_gR$ for all $g \in G$.
Define
\[
    \sigma_g(r) := u_g r u_g^{-1},
    \qquad
    \text{and}
    \qquad
    \omega(g,h) := u_g u_h u_{gh}^{-1}
\]
for all $g,h \in G$ and $r \in R$. 
The resulting maps 
\[
    \sigma: G \to \Aut(R)
    \qquad
    \text{and}
    \qquad
    \omega: G \times G \to \U(R)
\]
satisfy
\begin{align}
    \sigma_g\circ\sigma_h
    &=
    \Ad(\omega(g,h))\circ\sigma_{gh},
    \label{eq:crossed-action}
    \\
    \omega(g,h)\omega(gh,k)
    &=
    \sigma_g(\omega(h,k))\omega(g,hk)
    \label{eq:crossed-cocycle}
\end{align}
for all $g,h,k\in G$, together with $\sigma_e = \id_R$ and $\omega(e,g) = \omega(g,e) = 1$.
Such a pair $(\sigma,\omega)$ is called a \emph{normalized crossed system} for $(G,R)$.

Conversely, every normalized crossed system $(\sigma,\omega)$ for $(G,R)$ determines a strongly $G$-graded ring 
\[
    R\ast_{\sigma,\omega}G := \bigoplus_{g \in G} Ru_g
\]
with multiplication
\[
    (ru_g) (su_h)
    :=
    r \sigma_g(s) \omega(g,h)u_{gh}
\]
for all $g,h\in G$ and $r,s\in R$. 
This ring is the \emph{crossed product} associated with~$(\sigma,\omega)$.

Passing to outer automorphisms in \eqref{eq:crossed-action} gives a group homomorphism
\[
    \sigma: G \to \mathrm{Out}(R),
    \qquad
    g \mapsto [\sigma_g],
\]
which, for simplicity, is denoted by the same symbol.
For $\tau \in \Aut(R)$, let $R_\tau$ denote $R$ with its usual left action and right action
\[
    r \cdot a := r \tau(a).
\]
The assignment $\mathrm{Out}(R) \to \Pic(R)$, $[\tau] \mapsto [R_\tau]$ is an injective group homomorphism, and the Picard homomorphism of $S$ is the composite
\[
    G \xrightarrow{\;\sigma\;} \mathrm{Out}(R) \longrightarrow \Pic(R).
\]
A crossed-product $\wG$-lift of $S$ is obtained from a normalized crossed system $(\wsig,\wom)$ for $(\wG,R)$ satisfying $\wsig_g = \sigma_g$ and $\wom(g,h) = \omega(g,h)$ for all $g,h \in G$. 
Thus the crossed-product lifting problem asks whether $(\sigma,\omega)$ extends from $G$ to $\wG$.

A necessary first step is to extend the outer action to a homomorphism
\[
    \wsig: \wG \to \mathrm{Out}(R).
\]
Composing with the canonical embedding $\mathrm{Out}(R) \hookrightarrow \Pic(R)$ gives an extension
$\wpic: \wG \to \Pic(R)$ of the Picard homomorphism of $S$. 
This determines only the Picard classes of the additional homogeneous components. 
To obtain a crossed-product lift, one must also choose representatives $\wsig_x \in \Aut(R)$ and units $\wom(x,y)\in\U(R)$ satisfying \eqref{eq:crossed-action}--\eqref{eq:crossed-cocycle} and extending the original crossed system. 
By \cref{thm:obstruction}, such choices exist if and
only if
\[
    0 = \kappa_S(\wpic) \in H^3_{\wpic}(\wG,G;\U(Z(R))).
\]

We next describe equivalence in terms of crossed systems. 
Let $b: \wG \to \U(R)$ satisfy $b_g = 1$ for all $g \in G$. 
Replacing the homogeneous units by
\[
    u_x^b := b_x u_x,
    \qquad
    x \in \wG,
\]
changes the extended crossed system according to
\begin{align*}
    \wsig_x^{\,b}
    &=
    \Ad(b_x) \circ \wsig_x,
    \\
    \wom^{\,b}(x,y)
    &=
    b_x \wsig_x(b_y) \wom(x,y)b_{xy}^{-1}
\end{align*}
for all $x,y \in \wG$. 
Consequently, equivalence classes of crossed-product $\wG$-lifts of $S$ correspond to extensions of $(\sigma,\omega)$ modulo these transformations.

For a fixed extension $\wsig: \wG \to \mathrm{Out}(R)$ of $\overline{\sigma}$, let $\wpic: \wG \to \Pic(R)$ be the induced Picard homomorphism. 
Since its values are represented by twisted bimodules $R_{\wsig_x}$, $x \in \wG$, the lifts realizing $\wpic$ are crossed products. 
Thus \cref{thm:classification} recovers the fact that, when such lifts exist, their equivalence classes form a torsor under
\[
    H^2_{\wpic}(\wG,G;\U(Z(R))).
\]

If $R$ is commutative, then $\mathrm{Inn}(R)$ is trivial and hence $\mathrm{Out}(R) = \Aut(R)$. 
In this case $\sigma: G \to \Aut(R)$ is an action, and $\omega$ is a normalized $2$-cocycle for the induced action on $\U(R)$. 
Fix an extension $\wsig: \wG \to \Aut(R)$ of $\sigma$. 
A crossed-product $\wG$-lift of $S$ inducing $\wsig$ exists if and only if
\[
    [\omega]
    \in
    \operatorname{im}
    \!
    \left(
        H^2_{\wsig}(\wG,\U(R)) \to H^2_\sigma(G,\U(R))
    \right),
\]
where the map is induced by restriction. 
When such lifts exist, their equivalence classes form a torsor under
\[
    H^2_{\wsig}(\wG,G;\U(R)).
\]

\subsection{A Picard obstruction}

Even a crossed product need not admit a lift, since its outer action may fail
to extend to the overgroup. 
Let $\Bbbk = \mathbb F_\ell$ be a prime field and set $R := \Bbbk \times \Bbbk$.
The ordinary Picard group of $R$ is trivial, while $\Aut(R)$ is generated by the transposition $\tau:R \to R$, $\tau(a,b) := (b,a)$.
Consequently, the bimodule Picard group considered here is
\[
    \Pic(R) \cong \mathrm{Out}(R) \cong C_2,
\]
with nontrivial element represented by $R_\tau$.

Let $\wG =C_4 = \langle x\rangle$ and $G =\langle x^2\rangle \cong C_2$.
Let the nontrivial element $x^2\in G$ act on $R$ by $\tau$, and consider the crossed product $S :=R \rtimes_\tau G$.
Its Picard homomorphism $p: G \to \Pic(R)$ maps $x^2$ to the nontrivial element $[R_\tau]$.

Suppose that $p$ extended to a homomorphism $\wpic: C_4 \to \Pic(R)$.
Then
\[
    [R_\tau]
    =
    p(x^2)
    =
    \wpic(x^2)
    =
    \wpic(x)^2.
\]
This is impossible because every element of
$\Pic(R)\cong C_2$ has trivial square. Thus the Picard homomorphism of $S$
does not extend from $G$ to $\wG$, and Corollary~\ref{cor:picard-necessary} shows that $S$ admits no $C_4$-lift.

\subsection{Twisted group rings}

Twisted group rings give a particularly transparent special case of the crossed-product construction. 
Let $U := \U(Z(R))$ with the~trivial $\wG$-action, and let $\omega \in Z^2(G,U)$ be a normalized $2$-cocycle. 
The associated twisted group ring 
\[
    S : =R^\omega[G]
\]
is the free left $R$-module with basis $(u_g)_{g\in G}$ and multiplication
\[
    (ru_g)(su_h)
    :=
    rs\,\omega(g,h)u_{gh}
\]
for all $r,s\in R$ and $g,h\in G$. This is the crossed product associated
with the trivial action of $G$ on $R$ and the cocycle $\omega$. Its
homogeneous components are isomorphic to the standard $R$-bimodule $R$, so
its Picard homomorphism is trivial.

A $\wG$-lift with trivial Picard homomorphism is determined by a normalized
cocycle
\[
    \wom\in Z^2(\wG,U)
\]
whose restriction to $G\times G$ is $\omega$. 
Such a cocycle exists if and only if
\[
    [\omega]
    \in
    \operatorname{im}
    \!
    \left(
        H^2(\wG,U) \to H^2(G,U)
    \right).
\]
Indeed, an extension of the cohomology class can be modified by a coboundary so that its restriction agrees with $\omega$ itself.

In the long exact cohomology sequence of the pair $(\wG,G)$, the relative characteristic class is
\[
    \kappa_S(\wpic)
    =
    \delta[\omega]
    \in
    H^3(\wG,G;U),
\]
where $\delta:H^2(G,U) \to H^3(\wG,G;U)$ is the connecting homomorphism. 
Thus the vanishing of $\kappa_S(\wpic)$ is precisely the condition that $[\omega]$ extend to
$\wG$. 
When it vanishes, the equivalence classes of lifts with trivial Picard homomorphism form a torsor under
\[
    H^2(\wG,G;U).
\]

For $\omega = 1$, the inclusion $R[G] \subseteq R[\wG]$ gives the canonical lift of the ordinary group ring.

\subsection{Semidirect-product lifts}\label{sec:semi}

We conclude with an application of the normal-subgroup recognition theorem.
Let $Q$ be a group acting on $G$ and $S$ by group and ring automorphisms, respectively, and denote the latter action by
\[
    \beta: Q \to \Aut(S),
    \qquad
    q \mapsto \beta_q.
\]
Assume that the two actions are compatible with the grading, that is, $\beta_q(S_g) = S_{q \cdot g}$ for all $q \in Q$ and $g \in G$.
Form the semidirect product and skew group ring
\[
    \wG := G \rtimes Q
    \qquad
    \text{and}
    \qquad
    \wS := S \rtimes_\beta Q,
\]
respectively.
Writing $u_q$ for the canonical unit associated with $q \in Q$, multiplication
in $\wS$ is determined by $u_q s = \beta_q(s) u_q$ for all $q \in Q$ and $s \in S$.

Define
\[
    \wS_{(g,q)} := S_g u_q,
    \qquad
    (g,q) \in G \rtimes Q.
\]
For all $g,h\in G$ and $q,r\in Q$, one finds that
\[
    (S_gu_q) (S_hu_r)
    \subseteq
    S_g \beta_q(S_h) u_{qr}
    =
    S_{g(q\cdot h)} u_{qr}
    =
    \wS_{(g,q)(h,r)}.
\]
Thus
\[
    \wS
    =
    \bigoplus_{(g,q) \in G \rtimes Q} \wS_{(g,q)}
\]
is a $\wG$-graded ring. 
Its subgroup grading over $G\cong G\times\{e\}$ is the given grading of $S$, because $\wS_{(g,e)} = S_g$ for all $g \in G$.

The subgroup $G$ is normal in $G\rtimes Q$, with quotient isomorphic to $Q$.
The~homogeneous components of the corresponding quotient grading are
\[
    \wS_{[q]}
    =
    \bigoplus_{g \in G} S_gu_q
    =
    S u_q,
    \qquad
    q \in Q.
\]
Since every component $Su_q$ contains the unit $u_q$, the quotient grading is strong. 
It therefore follows from Corollary~\ref{cor:normal-lift} that $\wS$ is a $(G \rtimes Q)$-lift of $S$.

Although $\wS$ is a skew group ring over $S$, its $(G \rtimes Q)$-grading need not be a crossed-product grading over $R$. 
Indeed, the component $S_gu_q$ contains a unit if and only if $S_g$ does. 
Hence the lift constructed above is a crossed product precisely when the original strongly $G$-graded ring $S$ is a crossed product.

\section*{Acknowledgments}

The second author thanks Johan \"{O}inert for valuable discussions and
correspondence concerning this work.

\section*{Conflict of Interest}

The authors declare no conflict of interest.

\section*{Data Availability Statement}

Data sharing not applicable to this article as no datasets were generated or analysed during the current study.

\bibliographystyle{amsplain}
\bibliography{LSGRO}

@article{ClHaRi19,
 author = {Clark, Lisa Orloff and Hazrat, Roozbeh and Rigby, Simon W.},
 title = {Strongly graded groupoids and strongly graded {Steinberg} algebras},
 fjournal = {Journal of Algebra},
 journal = {J. Algebra},
 issn = {0021-8693},
 volume = {530},
 pages = {34--68},
 year = {2019},
 language = {English},
 doi = {10.1016/j.jalgebra.2019.03.030},
 zbMATH = {7060811},
 Zbl = {1444.16042}
}

@book{friedrich2000,
    author     = {Friedrich, Thomas},
    title      = {Dirac Operators in Riemannian Geometry},
    series     = {Graduate Studies in Mathematics},
    volume     = {25},
    publisher  = {American Mathematical Society},
    address    = {Providence, RI},
    year       = {2000},
    note       = {Translated by Andreas Nestke}
}

@book{MacLane95,
    author     = {Mac Lane, Saunders},
    title      = {Homology},
    series     = {Classics in Mathematics},
    publisher  = {Springer-Verlag},
    address    = {Berlin},
    year       = {1995},
    note       = {Reprint of the 1975 edition}
}

@book{Mont93,
    author     = {Montgomery, Susan},
    title      = {Hopf Algebras and Their Actions on Rings},
    series     = {CBMS Regional Conference Series in Mathematics},
    volume     = {82},
    publisher  = {American Mathematical Society},
    address    = {Providence, RI},
    year       = {1993}
}

@book{NuOy04,
    author     = {N{\u{a}}st{\u{a}}sescu, Constantin and
                  Van Oystaeyen, Freddy},
    title      = {Methods of Graded Rings},
    series     = {Lecture Notes in Mathematics},
    volume     = {1836},
    publisher  = {Springer-Verlag},
    address    = {Berlin},
    year       = {2004}
}

@mastersthesis{Hus26,
    author  = {Husen, Emma},
    title   = {{Hopf--Galois Theory as a Unifying Framework for Symmetry
               with a View towards Lifting Problems}},
    school  = {Linnaeus University},
    address = {V{\"a}xj{\"o}, Sweden},
    year    = {2026}
}

@phdthesis{Lann21,
    author  = {L{\"a}nnstr{\"o}m, Daniel},
    title   = {The Structure of Epsilon-Strongly Group Graded Rings},
    school  = {Blekinge Institute of Technology},
    address = {Karlskrona, Sweden},
    year    = {2021}
}

@article{Tak59,
    author  = {Takasu, Sat{\^{o}}},
    title   = {Relative Homology and Relative Cohomology Theory of Groups},
    journal = {J. Fac. Sci. Univ. Tokyo Sect. I},
    volume  = {8},
    pages   = {75--110},
    year    = {1959}
}

@misc{Wa25,
    author       = {Wagner, Stefan},
    title        = {Noncommutative Principal Bundles and Central Extensions},
    year         = {2025},
    note = {Preprint, arXiv:2509.02263},
    archivePrefix = {arXiv},
    primaryClass = {math.OA}
}

\end{document}